\documentclass[12pt]{article}
\usepackage{amssymb}
\usepackage{amsfonts}
\usepackage{graphicx}
\usepackage{amsmath}
\usepackage{hyperref}

\newtheorem{theorem}{Theorem}[section]

\newtheorem{assumption}[theorem]{Assumption}

\newtheorem{corollary}[theorem]{Corollary}

\newtheorem{definition}[theorem]{Definition}

\newtheorem{lemma}[theorem]{Lemma}

\newtheorem{proposition}[theorem]{Proposition}
\newtheorem{remark}[theorem]{Remark}

\newenvironment{proof}[1][Proof]{\textbf{#1.} }{\ \rule{0.5em}{0.5em}}

\numberwithin{equation}{section}

\begin{document}

\title{Functional differential equation with infinite delay in a space of
exponentially bounded and uniformly continuous functions}
\author{\textsc{Zhihua Liu$^{a,}$\thanks{Research was partially supported by National Natural Science Foundation of China (Grant Nos. 11871007 and 11811530272) and the Fundamental Research Funds for the Central Universities.} and Pierre Magal$^{b,}$\thanks{Research was partially supported by the French Ministry of Foreign and European Affairs program France-China Cai Yuanpei Campus France (30268SM).}} 
\\
%EndAName
$^{a}$\textit{\small School of Mathematical Sciences, Beijing Normal
University,}\\
\textit{\small Beijing 100875, People's Republic of China}\\
$^b${\small \textit{Univ. Bordeaux, IMB, UMR 5251, F-33400 Talence, France}} 
\\
{\small \textit{CNRS, IMB, UMR 5251, F-33400 Talence, France.}} }
\maketitle

\begin{abstract}
In this article we study a class of delay differential equations with
infinite delay in weighted spaces of uniformly continuous functions. We
focus on the integrated semigroup formulation of the problem and so doing we
provide a spectral theory. As a consequence we obtain a local stability
result and a Hopf bifurcation theorem for the semiflow generated by such a
problem.

\vspace{0.1cm} \noindent \textbf{Key words}. Functional differential
equations, infinite delay, integrated semigroup, stability, Hopf bifurcation

\vspace{0.1cm} \noindent \textbf{AMS Subject Classification}. 34K18, 34K20,
37L10.
\end{abstract}

\section{Introduction}

Functional differential equations with finite and infinite delay have been
extensively studied in the literature. Finite delay differential equations
have firstly been studied in the 1970s by the group of Hale's \cite%
{Hale71,Hale77, Hale-Lunel}. Since then people tried to extend some
bifurcation results for ordinary differential equations to functional
differential equations. In order to do so, one of the main difficulties is
to understand the relationship between the spectral properties of the
linearized system (around a given equilibrium) and the dynamical properties
of nonlinear perturbed systems. This type of questions have been directly
considered for delay differential equations by deriving a so called
variation of constant formula. We also refer to Arino and Sanchez \cite%
{Arino-Sanchez} and Kappel \cite{Kappel} for more results about this topic.

In the 1980s and 1990s variation of constant formula was reconsidered by
using non classical perturbation idea coming from semigroup theory. Sun-star
adjoint spaces and semigroup theory have been firstly successfully applied
to delay differential equations. We refer to Diekmann et al. \cite{Diekmann}
for a nice survey about this topic. We refer to Kaashoek and Verduyn Lunel 
\cite{Kaashoek-Verduyn Lunel}, Frasson and Verduyn Lunel \cite%
{Frasson-Verduyn-Lunel} and Diekmann Getto and Gyllenberg \cite%
{Diekmann-Getto-Gyllenberg} and references therein for more results in that
direction.

Around the same period Adimy \cite{Adimy90} and Thieme \cite{Thieme90a}
observed that integrated semigroups can also be used to derive a variation
of constant formula to describe functional differential equations. We also
refer to Adimy \cite{Adimy93a}, Adimy and Arino \cite{Adimy93b} and Ezzinbi
and Adimy \cite{Ezzinbi-Adimy} for more results about this topic. 
In the
present article, we will also use integrated semigroup theory and we will
reconsider the formulation introduced by Liu, Magal and Ruan \cite
{Liu-Magal-Ruan08} to study infinite delay differential equations. More recently, infinite delay has also been considered by Walther \cite{Walther} in the context of state dependent delay differential equations. 

Consider the weighted space of uniformly continuous functions 
\begin{equation*}
\begin{array}{r}
BUC_{\eta }=\left\{ \varphi \in C\left( \left( -\infty ,0\right] ,\mathbb{R}%
^{n}\right) :\text{ }\theta \rightarrow e^{\eta \theta }\varphi \left(
\theta \right) \text{ is bounded and}\right. \\ 
\left. \ \ \ \ \ \ \ \ \ \ \ \ \ \ \ \ \ \ \ \ \ \ \ \ \ \ \ \ \ \ \ \ \ 
\text{ uniformly continuous}\right\}%
\end{array}%
\end{equation*}%
which is a Banach space endowed with the norm 
\begin{equation*}
\left\Vert \varphi \right\Vert _{\eta }:=\sup_{\theta \leq 0}e^{\eta \theta
}\left\Vert \varphi \left( \theta \right) \right\Vert .
\end{equation*}

In this article we consider the following class of functional differential
equations on the space $BUC_{\eta }$ 
\begin{equation}
\text{(FDE)}\;\;\;\left\{ 
\begin{array}{l}
\dfrac{dx(t)}{dt}=f(x_{t}),\forall t\geq 0, \\ 
x_{0}=\varphi \in BUC_{\eta },%
\end{array}%
\right.  \label{1.1}
\end{equation}%
where $f:BUC_{\eta }\rightarrow \mathbb{R}^{n}$ is Lipschitz on bounded
sets. Recall that for any given map $x\in C\left( \left( -\infty ,\tau %
\right] ,\mathbb{R}^{n}\right) $ (for some $\tau \geq 0$) and each $t\leq
\tau $ the map $x_{t}\in C\left( \left( -\infty ,0\right] ,\mathbb{R}%
^{n}\right) $ is defined by 
\begin{equation*}
x_{t}\left( \theta \right) =x(t+\theta ),\forall \theta \leq 0.
\end{equation*}%
Then it is easy to verify that if $x\in C\left( \left( -\infty ,\tau \right]
,\mathbb{R}^{n}\right) $ for some $\tau \geq 0$ then

\begin{equation*}
x_0 \in BUC_{\eta } \Rightarrow x_t \in BUC_{\eta }, \forall t \in \left[0,
\tau \right].
\end{equation*}
Recall the notion of solution for the FDE.

\begin{definition}
\label{DE1.1}A solution of the FDE (\ref{1.1}) is a continuous map $x:\left(
-\infty ,\tau \right] \rightarrow \mathbb{R}^{n}$ (for some $\tau >0$)
satisfying 
\begin{equation}
x(t)=\left\{ 
\begin{array}{l}
\varphi (0)+\int_{0}^{t}f(x_{l})dl,\forall t\geq 0, \\ 
\varphi (t),\forall t\leq 0.%
\end{array}%
\right.  \label{1.2}
\end{equation}
\end{definition}
Assume that 
\begin{equation*}
f(0_{BUC_{\eta }})=0.
\end{equation*}%
Then $0$ is an equilibrium solution of the system (\ref{1.1}). Assume that $%
f $ is differential at $0$, and set 
\begin{equation*}
\widehat{L}:=Df(0).
\end{equation*}%
Then the linearized equation of (\ref{1.1}) around $0$ is 
\begin{equation}
\left\{ 
\begin{array}{l}
\dfrac{dx(t)}{dt}=\widehat{L}(x_{t}),\forall t\geq 0, \\ 
x_{0}=\varphi \in BUC_{\eta }.%
\end{array}%
\right.  \label{1.3}
\end{equation}
The first main question addressed in this article is to understand the
spectral properties of the linearized equation (\ref{1.3}). Then by using
the spectral properties of (\ref{1.3}) we will derive a stability and Hopf
bifurcation results for equation (\ref{1.1}). Actually the equation (\ref{1.1}) can be rewritten as 
\begin{equation}
\left\{ 
\begin{array}{l}
\dfrac{dx(t)}{dt}=\widehat{L}(x_{t})+g(x_{t}),\forall t\geq 0, \\ 
x_{0}=\varphi \in BUC_{\eta },
\end{array}%
\right.  \label{1.4}
\end{equation}%
where $g:=f-\widehat{L}$.

Assume that there exists $x\in C\left( \left( -\infty ,\tau \right] ,\mathbb{%
R}^{n}\right) $ (for some $\tau \geq 0$) a solution of (\ref{1.1}). Set 
\begin{equation}
u(t,\theta ):=x_{t}(\theta )=x(t+\theta ),\forall t\in \lbrack 0,\tau ]\text{
and }\forall \theta \leq 0.  \label{1.5}
\end{equation}%
Then $u$ can be regarded as a solution of the following system 
\begin{equation}
\text{(PDE)}\;\;\;\left\{ 
\begin{array}{l}
\partial _{t}u(t,\theta )-\partial _{\theta }u(t,\theta )=0,\text{ for }%
\theta \leq 0\text{ and }t\geq 0, \\ 
\partial _{\theta }u(t,0)=f(u(t,.)),\text{ for }t\geq 0, \\ 
u\left( 0,.\right) =\varphi \in BUC_{\eta }.%
\end{array}%
\right.  \label{1.6}
\end{equation}

The idea of using the PDE associated to the FDE, was successufully used by
Travis and Webb \cite{Travis-Webb74,Travis-Webb78} and Webb \cite{Webb76} to
apply nonlinear semigroup theory. We refer to Ruess \cite{Ruess} for more
results and updated references on this topic. To our best knowledge no
bifurcation results have been obtained by using this approach.

In order to reformulate the PDE problem as an abstract Cauchy problem, we
will first incorporate the boundary condition into the state space, by
considering the Banach space 
\begin{equation*}
X:=\mathbb{R}^{n}\times BUC_{\eta }
\end{equation*}%
endowed with the product norm 
\begin{equation*}
\left\Vert \left( \alpha ,\varphi \right) \right\Vert :=\left\Vert \alpha
\right\Vert _{\mathbb{R}^{n}}+\left\Vert \varphi \right\Vert _{\eta }.
\end{equation*}%
Define $A:D(A)\subset X\rightarrow X$ the linear operator by 
\begin{equation}
A\left( 
\begin{array}{c}
0_{\mathbb{R}^{n}} \\ 
\varphi%
\end{array}%
\right) :=\left( 
\begin{array}{c}
-\varphi ^{\prime }(0) \\ 
\varphi ^{\prime }%
\end{array}%
\right) ,\forall \left( 
\begin{array}{c}
0_{\mathbb{R}^{n}} \\ 
\varphi%
\end{array}%
\right) \in D(A),  \label{1.7}
\end{equation}%
with 
\begin{equation*}
D(A)=\left\{ 0_{\mathbb{R}^{n}}\right\} \times BUC_{\eta }^{1}
\end{equation*}%
where 
\begin{equation*}
BUC_{\eta }^{1}:=\left\{ \varphi \in C^{1}\left( \left( -\infty ,0\right] ,%
\mathbb{R}^{n}\right) :\varphi ,\varphi ^{\prime }\in BUC_{\eta }\left(
\left( -\infty ,0\right] ,\mathbb{R}^{n}\right) \right\} .
\end{equation*}%
It is important to note that the closure of the domain is 
\begin{equation*}
X_{0}:=\overline{D(A)}=\left\{ 0_{\mathbb{R}^{n}}\right\} \times BUC_{\eta }.
\end{equation*}%
Therefore, $A$ is non-densely defined. We also define $L:\overline{D(A)}%
\rightarrow X$ by 
\begin{equation*}
L\left( 
\begin{array}{c}
0_{\mathbb{R}^{n}} \\ 
\varphi%
\end{array}%
\right) =\left( 
\begin{array}{c}
\widehat{L}\left( \varphi \right) \\ 
0_{BUC_{\eta }}%
\end{array}%
\right) .
\end{equation*}
We consider $F:\overline{D(A)}\rightarrow X$ and $G:\overline{D(A)}%
\rightarrow X$ the maps defined by 
\begin{equation*}
F\left( 
\begin{array}{c}
0_{\mathbb{R}^{n}} \\ 
\varphi%
\end{array}%
\right) =\left( 
\begin{array}{c}
f(\varphi ) \\ 
0_{BUC_{\eta }}%
\end{array}%
\right) \text{ and }G\left( 
\begin{array}{c}
0_{\mathbb{R}^{n}} \\ 
\varphi%
\end{array}%
\right) =\left( 
\begin{array}{c}
g(\varphi ) \\ 
0_{BUC_{\eta }}%
\end{array}%
\right) .
\end{equation*}%
Then 
\begin{equation*}
F=L+G.
\end{equation*}%
Set 
\begin{equation*}
v(t)=\left( 
\begin{array}{c}
0_{\mathbb{R}^{n}} \\ 
u(t)%
\end{array}%
\right) .
\end{equation*}%
As we will see in section 2, the FDE (\ref{1.1}) or the PDE (\ref{1.6}) can
be reformulated as the following abstract non-densely defined Cauchy problem 
\begin{equation}
\text{(ACP)}\left\{ 
\begin{array}{l}
\dfrac{dv(t)}{dt}=Av(t)+L(v(t))+G(v(t)),t\geq 0, \\ 
v(0)=\left( 
\begin{array}{c}
0_{\mathbb{R}^{n}} \\ 
\varphi%
\end{array}%
\right) \in \overline{D(A)}.%
\end{array}%
\right.  \label{1.8}
\end{equation}

Several examples of infinite delay differential equations have been
considered in the literature. Here we present two examples. One is the
following chemostat model with a distributed delay considered in Ruan and
Wolkowicz \cite{Ruan-Wolkowicz}
\begin{equation}
\left\{ 
\begin{array}{l}
\frac{dS}{dt}=(S^{0}-S(t))D-ax(t)p(S(t)), \\ 
\frac{dx}{dt}=x(t)\left[ -D_{1}+\int_{-\infty }^{t}F(t-\tau )p(S(\tau
))d\tau \right] , \\ 
S(s)=\phi (s)>0,-\infty <s\leq 0,x(0)=x_{0}>0, \\ 
\phi (s)\text{ is a continuous function on }(-\infty ,0],%
\end{array}%
\right.  \label{1.9}
\end{equation}%
where $S(t)$ and $x(t)$ denote the concentration of the nutrient and the
populations of microorganisms at time $t$. Another example is the following
model of a fishery with fish stock involving delay equations investigated by
Auger and Ducrot \cite{Auger-Ducrot} 
\begin{equation}
\left\{ 
\begin{array}{l}
\frac{dn}{dt}=rn(1-n)-\varphi (n,E), \\ 
\frac{dE}{dt}=p(1-\eta )\varphi (n,E)+\eta p\int_{-\infty }^{0}\delta
e^{\delta \theta }\varphi (n(t+\theta ),E(t+\theta ))d\theta -cE.%
\end{array}%
\right.  \label{1.10}
\end{equation}%
where $n(t),E(t)$ denote the density of the resource and the fishing effort,
respectively. We refer to McCluskey \cite{McCluskey}, Rost and Wu \cite%
{Rost-Wu}, and Gourley, Rost and Thieme \cite{Gourley-Rost-Thieme} for more
examples in the context of population dynamics. We would like to mension
that we can derive a stability and Hopf bifurcation result for the models (%
\ref{1.9}) and (\ref{1.10}) by using the results presented in this article.

The last example is a model describing the interaction between floating structures in shallow water. More precisely, in order to describe the movement of the surface of the water Bocchi \cite{Bocchi} derives the following equation 
$$
f \left( \delta(t)\right) \ddot{\delta}(t)=\int_0^t F(s) \dot{\delta}(t-s)ds+g(\delta(t), \dot{\delta}(t)) 
$$
where $f$ and $g$ are smooth functions and the range of $f$ does not contain $0$.  The map $t \to F(t)$ is a continuous function which converges exponentially fast to $0$ when $t$ goes to infinity. Our results apply to this class of equation and we refer to \cite{Bocchi} for more results. 

The main tool of this article is to apply integrated semigroup theory. We
use essentially the results in Thieme \cite{Thieme90a,Thieme90b, Thieme97},
Magal and Ruan \cite{Magal-Ruan07, Magal-Ruan09b} and Liu, Magal and Ruan 
\cite{Liu-Magal-Ruan08}. We will first prove that $A$ is a Hille-Yosida
operator in order to define the mild solution of ACP. By using the
uniqueness of the mild solution of ACP, we will prove that each mild
solution of ACP corresponds to a solution of the FDE and the reverse is also
true. Then in order to obtain a spectral theory, we study the essential
growth rate of the semigroup generated by $A_{0}$  and by $(A+L)_{0}$. The operators $A_{0}$ (respectively $(A+L)_{0}$) is the part of $A$ (respectively $A+L$) in $\overline{D(A)}$ (see section 4). This part is crucial to understand the spectral properties of these semigroup.
Actually the fact that $\eta >0$ is crucial to obtain a stability results,
as well as a center manifold theorem (see Magal and Ruan \cite{Magal-Ruan09a}%
) and a Hopf bifurcation theorem. We refer to the book of Magal and Ruan \cite{Magal-Ruan18} for a nice survey on this topic. 

By transforming the system it is also possible to study the problem 
\begin{equation}
BUC:=BUC_{0}.  \label{1.11}
\end{equation}%
Consider the isometry form $\Psi :BUC_{\eta }\rightarrow BUC$ defined by 
\begin{equation*}
\Psi (u)(\theta ):=e^{\eta \theta }u(\theta ).
\end{equation*}%
By setting $\widehat{u}(t,\theta ):=e^{\eta \theta }u(t,\theta )$, the PDE (%
\ref{1.6}) gives 
\begin{equation}
\text{(PDE)}\;\;\;\left\{ 
\begin{array}{l}
\partial _{t}\widehat{u}(t,\theta )-\partial _{\theta }\widehat{u}(t,\theta
)=-\eta \widehat{u}(t,\theta ),\text{ for }\theta \leq 0\text{ and }t\geq 0,
\\ 
\partial _{\theta }\widehat{u}(t,0)-\eta \widehat{u}(t,0)=f(e^{-\eta .}%
\widehat{u}(t,.)),\text{ for }t\geq 0, \\ 
\widehat{u}\left( 0,.\right) =\varphi \in BUC.%
\end{array}%
\right.  \label{1.12}
\end{equation}%
Identifying $\widehat{u}(t,.)$ and $\widehat{v}(t)=\left( 
\begin{array}{c}
0_{\mathbb{R}^{n}} \\ 
\widehat{u}\left( t,.\right)%
\end{array}%
\right) $ the last PDE can be rewritten as an abstract Cauchy problem 
\begin{equation}
\dfrac{d\widehat{v}(t)}{dt}=B\widehat{v}(t)+H(\widehat{v}(t)),t\geq 0,%
\widehat{v}(0)=\widehat{v}_{0}\in \overline{D(B)}  \label{1.13}
\end{equation}%
where $B:D(B)\subset Y\rightarrow Y$ (where the Banach $Y:=\mathbb{R}%
^{n}\times BUC$ is endowed with the usual product norm) is the linear
operator defined by 
\begin{equation*}
B\left( 
\begin{array}{c}
0_{\mathbb{R}^{n}} \\ 
\varphi%
\end{array}%
\right) :=\left( 
\begin{array}{c}
-\varphi ^{\prime }(0)+\eta \varphi (0) \\ 
\varphi ^{\prime }-\eta \varphi%
\end{array}%
\right) ,\forall \left( 
\begin{array}{c}
0_{\mathbb{R}^{n}} \\ 
\varphi%
\end{array}%
\right) \in D(B),
\end{equation*}%
with 
\begin{equation*}
D(B)=\left\{ 0_{\mathbb{R}^{n}}\right\} \times BUC^{1}
\end{equation*}%
and $BUC^{1}:=BUC_{0}^{1}$. $H:\overline{D(B)}\rightarrow X$ is the map
defined by 
\begin{equation*}
H\left( 
\begin{array}{c}
0_{\mathbb{R}^{n}} \\ 
\varphi%
\end{array}%
\right) =\left( 
\begin{array}{c}
f(e^{-\eta .}\varphi (.)) \\ 
0_{BUC}%
\end{array}%
\right) .
\end{equation*}

For infinite delay differential equations various class of semi-normed
spaces have been considered firstly by Hale and Kato \cite{Hale-Kato}. We
refer to the book Hino, Murakami and Naito \cite{Hino-Murakami-Naito} for
more results and a nice survey on this subject. Along this line a variation
of constant formula has been obtained by Hino, Murakami, Naito and Minh \cite%
{Hino-Murakami-Naito-Minh}. We also refer to Diekmann and Gyllenberg \cite%
{Diekmann-Gyllenberg} for infinite delay differential equations in weighted $%
L^{1}$ space. Along this line Matsunaga, Murakami, Nagabuchi and Van Minh 
\cite{Matsunaga-Murakami-Nagabuchi-Van-Minh} recently proved a center
manifold theorem for difference equation in $L^{1}$ space. We shall also
mention that a Hopf bifurcation theorem has been obtained by Hassard,
Kazarinoff and Wan \cite[Chapter 4 Section 5]{Hassard-Kazarinoff-Wan} in $%
L^{2}$ function space. 

This article is entirely devoted to the first class of problem in $BUC_{\eta }$.
The paper is organized as follows. In section 2 we study the properties of
the linear operator $A$. In order to obtain an explicit formula for the
integrated solution of the abstract Cauchy problem (\ref{1.8}) we firstly
consider a special case of (\ref{1.8}) in section 3. In section 4, explicit
formulas for some mild solutions are given and the properties of the linear
operator $A+L$ are investigated. In section 5, we obtain an explicit formula
for the projectors on the generalized eigenspaces associated to some
eigenvalues. The projector for a simple eigenvalue is considered in section
6. Sections 7 and 8 deal with the nonlinear semiflow and local stability of
equilibria respectively. In Section 9, we show a few comments and remarks on
the center manifold theorem, Hopf bifurcation theory and normal form theory
for infinite delay differential equations.

%\section{Abstract theory}

\section{Preliminary results}

In order to apply integrated semigroup theory we need to verify the Hille-Yosida 
properties for the linear operator $A$.

\begin{lemma}
\label{LE2.1} We have $(0,+\infty ) \subset \rho \left( A\right)$ (where $\rho \left( A\right)$ is the resolvent set of $A$), and we have for each $\lambda >0$ and each $\left( 
\begin{array}{c}
\alpha \\ 
\varphi%
\end{array}%
\right) \in X$
\begin{equation}
\begin{array}{l}
\left( \lambda I-A\right) ^{-1}\left( 
\begin{array}{c}
\alpha \\ 
\varphi%
\end{array}%
\right) =\left( 
\begin{array}{c}
0_{\mathbb{R}^{n}} \\ 
\psi%
\end{array}%
\right) \\ 
\Leftrightarrow \psi (\theta )=\frac{1}{\lambda }e^{\lambda \theta }\left[
\alpha +\varphi \left( 0\right) \right] +\int_{\theta }^{0}e^{\lambda
(\theta -l)}\varphi \left( l\right) dl.%
\end{array}
\label{2.1}
\end{equation}
\end{lemma}
\begin{proof}
Let $\lambda \in (0,\infty )$. For $\left( 
\begin{array}{c}
\alpha \\ 
\varphi%
\end{array}%
\right) \in X$ and $\left( 
\begin{array}{c}
0_{\mathbb{R}^{n}} \\ 
\psi%
\end{array}%
\right) \in D(A),$ we have 
\begin{equation*}
\begin{array}{l}
\left( \lambda I-A\right) \left( 
\begin{array}{c}
0_{\mathbb{R}^{n}} \\ 
\psi%
\end{array}%
\right) =\left( 
\begin{array}{c}
\alpha \\ 
\varphi%
\end{array}%
\right) \Leftrightarrow \left\{ 
\begin{array}{c}
\psi ^{\prime }(0)=\alpha \\ 
\lambda \psi -\psi ^{\prime }=\varphi%
\end{array}%
\right. \\ 
\Leftrightarrow \left\{ 
\begin{array}{l}
\lambda \psi (0)=\alpha +\varphi \left( 0\right) \\ 
\lambda \psi -\psi ^{\prime }=\varphi%
\end{array}%
\right. \\ 
\Leftrightarrow \left\{ 
\begin{array}{l}
\lambda \psi (0)=\alpha +\varphi \left( 0\right) \\ 
\psi \left( \theta \right) =e^{\lambda (\theta -\widehat{\theta })}\psi
\left( \widehat{\theta }\right) +\int_{\widehat{\theta }}^{\theta
}e^{\lambda (\theta -l)}\varphi \left( l\right) dl,\forall \theta \geq 
\widehat{\theta }%
\end{array}%
\right. \\ 
\Leftrightarrow \left\{ 
\begin{array}{l}
\lambda \psi (0)=\alpha +\varphi \left( 0\right) \\ 
\psi \left( \widehat{\theta }\right) =e^{\lambda \widehat{\theta }}\psi
\left( 0\right) -\int_{0}^{\widehat{\theta }}e^{\lambda (\widehat{\theta }%
-l)}\varphi \left( l\right) dl,\forall \widehat{\theta }\leq 0%
\end{array}%
\right. \\ 
\Leftrightarrow \psi \left( \widehat{\theta }\right) =\frac{1}{\lambda }%
e^{\lambda \widehat{\theta }}\left[ \alpha +\varphi \left( 0\right) \right]
-\int_{0}^{\widehat{\theta }}e^{\lambda (\widehat{\theta }-l)}\varphi \left(
l\right) dl,\forall \widehat{\theta }\leq 0.%
\end{array}%
\end{equation*}
\end{proof}
\begin{lemma}
\label{LE2.2} The linear operator $A:D(A)\subset X\rightarrow X$ is a
Hille-Yosida operator.\ More precisely, we have 
\begin{equation}
\left\Vert \left( \lambda I-A\right) ^{-n}\right\Vert _{\mathcal{L}\left(
X\right) }\leq \frac{1}{\lambda ^{n}},\forall n\geq 1,\forall \lambda >0.
\label{2.2}
\end{equation}
\end{lemma}
\begin{proof}
Using (\ref{2.1}), we obtain 
\begin{eqnarray*}
&&\left\Vert \left( \lambda I-A\right) ^{-1}\left( 
\begin{array}{c}
\alpha \\ 
\varphi%
\end{array}%
\right) \right\Vert \\
&\leq &\sup_{\theta \leq 0}\left[ e^{\eta \theta }e^{\lambda \theta
}\left\vert \frac{1}{\lambda }\left[ \varphi \left( 0\right) +\alpha \right]
\right\vert +e^{\eta \theta }\int_{\theta }^{0}e^{\lambda \left( \theta
-s\right) }\left\vert \varphi \left( s\right) \right\vert ds\right] \\
&\leq &\frac{1}{\lambda }\left\vert \alpha \right\vert +\sup_{\theta \leq 0}%
\left[ \left( \frac{e^{\eta \theta }e^{\lambda \theta }}{\lambda }+\frac{%
(1-e^{\eta \theta }e^{\lambda \theta })}{\lambda }\right) \left\Vert \varphi
\right\Vert _{\eta }\right] \\
&\leq &\frac{1}{\lambda }\left[ \left\vert \alpha \right\vert +\left\Vert
\varphi \right\Vert _{\eta }\right] \\
&=&\frac{1}{\lambda }\left\Vert \left( 
\begin{array}{c}
\alpha \\ 
\varphi%
\end{array}%
\right) \right\Vert .
\end{eqnarray*}%
Therefore, (\ref{2.2}) holds and the proof is completed.
\end{proof}
\begin{lemma}
\label{LE2.3}%
\begin{equation*}
\overline{D(A)}=\left\{ 0\right\} \times BUC_{\eta }.
\end{equation*}
\end{lemma}
\begin{proof}
Let $\psi \in BUC_{\eta }\left( \left( -\infty ,0\right] ,\mathbb{R}%
^{n}\right) .$ Define for each $\varepsilon >0$ and each $\theta \leq 0$ 
\begin{equation*}
\psi _{\varepsilon }(\theta )=e^{-\eta \theta }\frac{1}{\varepsilon }%
\int_{\theta -\varepsilon }^{\theta }e^{\eta l}\psi \left( l\right) dl.
\end{equation*}%
Then since $l\rightarrow e^{\eta l}\psi \left( l\right) $ is bounded and
uniformly continuous we deduce that for each $\varepsilon >0$%
\begin{equation*}
\psi _{\varepsilon }\in BUC_{\eta }^{1}
\end{equation*}%
and 
\begin{equation*}
\lim_{\varepsilon \rightarrow 0}\left\Vert \psi _{\varepsilon }-\psi
\right\Vert _{\eta }=0\text{,}
\end{equation*}%
the proof is complete.\ 
\end{proof}

\section{A special class of mild solutions}

In this section in order to obtain an explicit formula for the integrated
solution of the abstract Cauchy problem (\ref{1.8}) we firstly consider the
following PDE 
\begin{equation}
\left\{ 
\begin{array}{l}
\partial _{t}u(t,\theta )-\partial _{\theta }u(t,\theta )=0,\text{ for }%
\theta \leq 0\text{ and }t\geq 0 \\ 
\partial _{\theta }u(t,0)=h(t),\text{ for }t\geq 0 \\ 
u\left( 0,.\right) =\varphi \in BUC_{\eta }%
\end{array}%
\right.  \label{3.1}
\end{equation}%
where the map $h\in L^{1}((0,\tau );\mathbb{R}^{n})$ is a given perturbation
at the boundary.

In that case the system (\ref{1.8}) becomes 
\begin{equation}
\frac{dv(t)}{dt}=Av(t)+\left( 
\begin{array}{c}
h(t) \\ 
0%
\end{array}%
\right) ,t\geq 0,\text{ }v(0)=\left( 
\begin{array}{c}
0_{\mathbb{R}^{n}} \\ 
\varphi%
\end{array}%
\right) \in \overline{D(A)},  \label{3.2}
\end{equation}%
where $h\in L^{1}\left( \left( 0,\tau \right) ,\mathbb{R}^{n}\right) $.

Recall that $v\in C\left( \left[ 0,\tau \right] ,X\right) $ is an integrated
solution of (\ref{3.2}) if and only if 
\begin{equation}
\int_{0}^{t}v(s)ds\in D(A),\forall t\in \left[ 0,\tau \right]  \label{3.3}
\end{equation}%
and 
\begin{equation}
v(t)=\left( 
\begin{array}{c}
0_{\mathbb{R}^{n}} \\ 
\varphi%
\end{array}%
\right) +A\int_{0}^{t}v(s)ds+\int_{0}^{t}\left( 
\begin{array}{c}
h(s) \\ 
0%
\end{array}%
\right) ds.  \label{3.4}
\end{equation}%
We refer to Arendt \cite{Arendt87b)}, Thieme \cite{Thieme90b}, Kellermann
and Hieber \cite{Kellermann}, and the book by Arendt et al.\ \cite{Arendt01}
for a nice overview on this subject.\ We also refer to Magal and Ruan \cite%
{Magal-Ruan18} for more results and updated references.\ 

From (\ref{3.2}) we note that if $v$ is an integrated solution we must have 
\begin{equation*}
v(t)=\lim_{h\rightarrow 0^{+}}\frac{1}{h}\int_{t}^{t+h}v(s)ds\in \overline{%
D(A)}.
\end{equation*}%
Hence 
\begin{equation*}
v(t)=\left( 
\begin{array}{c}
0_{\mathbb{R}^{n}} \\ 
u(t)%
\end{array}%
\right)
\end{equation*}%
with 
\begin{equation*}
u\in C\left( \left[ 0,\tau \right] ,BUC_{\eta }\right) .
\end{equation*}%
In order to obtain the uniqueness of the integrated solutions of (\ref{3.2})
we want to prove that $A$ generates an integrated semigroup.\ So firstly we
need to study the resolvent of $A.$ Since $A$ is a Hille-Yosida operator, $A$
generates a non-degenerated integrated semigroup $\left\{ S_{A}(t)\right\}
_{t\geq 0}$ on $X.\ $It follows from Thieme \cite{Thieme90b}, \ and
Kellerman and Hieber that the abstract Cauchy problem (\ref{3.2}) has at
most one integrated solution.

\begin{lemma}
\label{LE3.1} Let $h\in L^{1}\left( \left( 0,\tau \right) ,\mathbb{R}%
^{n}\right) $ and $\varphi \in BUC_{\eta }$. Then there exists $t\rightarrow
v(t)$ a unique integrated solution of the Cauchy problem (\ref{3.2}).
Moreover $v(t)$ is explicitly given by the following formula 
\begin{equation*}
v(t)=\left( 
\begin{array}{c}
0_{\mathbb{R}^{n}} \\ 
u(t)%
\end{array}%
\right)
\end{equation*}%
with 
\begin{equation}
u(t)\left( \theta \right) =x(t+\theta ),\forall t\in \left[ 0,\tau \right]
,\forall \theta \leq 0,  \label{3.5}
\end{equation}%
where 
\begin{equation*}
x(t)=\left\{ 
\begin{array}{l}
\varphi \left( 0\right) +\int_{0}^{t}h(s)ds,\text{ if }t\in \left[ 0,\tau %
\right] , \\ 
\varphi \left( t\right) ,\text{ if }t\leq 0.%
\end{array}%
\right.
\end{equation*}
\end{lemma}
\begin{proof}
Since $A$ is a Hille-Yosida operator, there is at most one integrated
solution of the Cauchy problem (\ref{3.2}).\ So it is sufficient to prove
that $u$ defined by (\ref{3.5}) satisfies for each $t\in \left[ 0,\tau %
\right] $ the following 
\begin{equation}
\left( 
\begin{array}{c}
0_{\mathbb{R}^{n}} \\ 
\int_{0}^{t}u(l)dl%
\end{array}%
\right) \in D(A)  \label{3.6}
\end{equation}%
and 
\begin{equation}
\left( 
\begin{array}{c}
0_{\mathbb{R}^{n}} \\ 
u(t)%
\end{array}%
\right) =\left( 
\begin{array}{c}
0_{\mathbb{R}^{n}} \\ 
\varphi%
\end{array}%
\right) +A\left( 
\begin{array}{c}
0_{\mathbb{R}^{n}} \\ 
\int_{0}^{t}u(l)dl%
\end{array}%
\right) +\left( 
\begin{array}{c}
\int_{0}^{t}h(l)dl \\ 
0%
\end{array}%
\right) .  \label{3.7}
\end{equation}%
Since 
\begin{equation*}
\int_{0}^{t}u(l)\left( \theta \right) dl=\int_{0}^{t}x(l+\theta
)dl=\int_{\theta }^{t+\theta }x(s)ds
\end{equation*}%
and therefore $\int_{0}^{t}u(l)dl\in BUC_{\eta }^{1}$ and (\ref{3.6})
follows. Moreover 
\begin{eqnarray*}
A\left( 
\begin{array}{c}
0 \\ 
\int_{0}^{t}u(l)dl%
\end{array}%
\right) &=&\left( 
\begin{array}{c}
-\left( x(t)-x(0)\right) \\ 
\left( x(t+.)-x(.)\right)%
\end{array}%
\right) \\
&=&-\left( 
\begin{array}{c}
0 \\ 
\varphi%
\end{array}%
\right) +\left( 
\begin{array}{c}
-\left( x(t)-\varphi (0)\right) \\ 
x(t+.)%
\end{array}%
\right) .
\end{eqnarray*}%
Therefore, (\ref{3.7}) is satisfied if and only if 
\begin{equation}
x(t)=\varphi (0)+\int_{0}^{t}h(s)ds.  \label{3.8}
\end{equation}%
The proof is completed.
\end{proof}

%As a consequence we obtain the following correspondence between the solutions of the functional differential equation and the solution of the abstract Cauchy problem.   
%\begin{proposition} A continuous map $x:\left(
%-\infty ,\tau \right] \rightarrow \mathbb{R}^{n}$ (for some $\tau >0$) is a solution of the FDE (\ref{1.1}) if and only if 
%$$
%v(t):=\left(
%\begin{array}{l}
%0_{\mathbb{R}^n}\\
%x_t
%\end{array}
%\right)
%$$
%is the unique integrated solution of the abstract Cauchy problem \eqref{1.8}. 
%\end{proposition} 
\section{Linear abstract Cauchy problem}

\bigskip For a given bounded linear operator $L\in \mathcal{L}\left(
X\right) ,$ \textrm{\ }$\left\Vert L\right\Vert _{ess}$ is the essential
norm of $L$ defined by 
\begin{equation*}
\left\Vert L\right\Vert _{ess}=\kappa \left( L\left( B_{X}\left( 0,1\right)
\right) \right) ,
\end{equation*}%
here $B_{X}\left( 0,1\right) =\left\{ x\in X:\left\Vert x\right\Vert
_{X}\leq 1\right\} ,$ and for each bounded set $B\subset X,$ $\kappa \left(
B\right) =\inf \left\{ \varepsilon >0:B\text{ can be covered by a finite
number of balls of radius }\leq \varepsilon \right\} $ is the Kuratovsky
measure of non-compactness. Let $L:D(L)\subset X\rightarrow X$ be the
infinitesimal generator of a linear $C_{0}$-semigroup $\left\{
T_{L}(t)\right\} _{t\geq 0}$ on a Banach space $X.$ Define the growth bound $%
\omega _{0}\left( L\right) \in \lbrack -\infty ,+\infty )$ of $L$ by 
\begin{equation*}
\omega _{0}\left( L\right) :=\lim_{t\rightarrow +\infty }\frac{\ln \left(
\left\Vert T_{L}(t)\right\Vert _{\mathcal{L}\left( X\right) }\right) }{t}.
\end{equation*}%
The essential growth bound $\omega _{0,ess}\left( L\right) \in \left[
-\infty ,+\infty \right) $ of\ $L$ is defined by 
\begin{equation*}
\omega _{0,ess}\left( L\right) :=\lim_{t\rightarrow +\infty }\frac{\ln
\left( \left\Vert T_{L}(t)\right\Vert _{ess}\right) }{t}.
\end{equation*}%
Recall that $A_{0}:D\left( A_{0}\right) \subset \overline{D(A)}\rightarrow 
\overline{D(A)}$ the part of $A$ in $\overline{D(A)}$ is defined by 
\begin{equation*}
A_{0}\left( 
\begin{array}{c}
0_{\mathbb{R}^{n}} \\ 
\varphi%
\end{array}%
\right) =\left( 
\begin{array}{c}
0_{\mathbb{R}^{n}} \\ 
\varphi ^{\prime }%
\end{array}%
\right) ,\forall \left( 
\begin{array}{c}
0_{\mathbb{R}^{n}} \\ 
\varphi%
\end{array}%
\right) \in D\left( A_{0}\right) ,
\end{equation*}%
where 
\begin{equation*}
D\left( A_{0}\right) =\left\{ \left( 
\begin{array}{c}
0_{\mathbb{R}^{n}} \\ 
\varphi%
\end{array}%
\right) \in \left\{ 0_{\mathbb{R}^{n}}\right\} \times BUC_{\eta
}^{1}:\varphi ^{\prime }(0)=0\right\} .
\end{equation*}%
Now by using the fact that $A$ is a Hille-Yosida operator, we deduce that $%
A_{0}$ is the infinitesimal generator of a strongly continuous semigroup $%
\left\{ T_{A_{0}}(t)\right\} _{t\geq 0}$ and $v(t)=T_{A_{0}}(t)\left( 
\begin{array}{c}
0_{\mathbb{R}^{n}} \\ 
\varphi%
\end{array}%
\right) $ is an integrated solution of 
\begin{equation*}
\frac{dv(t)}{dt}=Av(t),t\geq 0,\text{ }v(0)=\left( 
\begin{array}{c}
0_{\mathbb{R}^{n}} \\ 
\varphi%
\end{array}%
\right) \in \overline{D(A)}.
\end{equation*}%
Using Lemma \ref{LE3.1} with $h=0,$ we obtain the following result.\ 

\begin{lemma}
\label{LE4.1}The linear operator $A_{0}$ is the infinitesimal generator of a
strongly continuous semigroup $\left\{ T_{A_{0}}(t)\right\} _{t\geq 0}$ of
bounded linear operators on $\overline{D(A)}$ which is defined by 
\begin{equation}
T_{A_{0}}(t)\left( 
\begin{array}{c}
0_{\mathbb{R}^{n}} \\ 
\varphi%
\end{array}%
\right) =\left( 
\begin{array}{c}
0_{\mathbb{R}^{n}} \\ 
\widehat{T}_{A_{0}}(t)\varphi%
\end{array}%
\right) ,  \label{4.1}
\end{equation}%
where%
\begin{equation*}
\widehat{T}_{A_{0}}(t)(\varphi )(\theta )=\left\{ 
\begin{array}{l}
\varphi (0),\text{ if }t+\theta \geq 0, \\ 
\varphi (t+\theta ),\text{ if }t+\theta \leq 0.%
\end{array}%
\right.
\end{equation*}
\end{lemma}

The semigroup $\left\{ T_{A_{0}}(t)\right\} _{t\geq 0}$ can be rewritten as
follows 
\begin{equation}
T_{A_{0}}(t)\left( 
\begin{array}{c}
0_{\mathbb{R}^{n}} \\ 
\varphi%
\end{array}%
\right) =\left( 
\begin{array}{c}
0_{\mathbb{R}^{n}} \\ 
T\varphi +S(t)(\varphi )%
\end{array}%
\right) ,  \label{4.2}
\end{equation}%
where $T\varphi =\varphi (0)$ and 
\begin{equation*}
S(t)(\varphi )(\theta )=\left\{ 
\begin{array}{l}
0,\text{ if }t+\theta \geq 0, \\ 
\varphi (t+\theta )-\varphi (0),\text{ if }t+\theta \leq 0.%
\end{array}%
\right.
\end{equation*}%
Note that $T$ is a finite-rank operator and thus compact. For each $t\geq 0,$
we have 
\begin{equation*}
\sup_{\theta \leq -t}e^{\eta \theta }\left\Vert S(t)(\varphi )(\theta
)\right\Vert \leq 2e^{-\eta t}\left\Vert \varphi \right\Vert _{\eta }.
\end{equation*}%
Thus we obtain the following lemma.

\begin{lemma}
\label{LE4.2}The essential growth bound of $A_{0}$ satisfies%
\begin{equation*}
\omega _{0,ess}(A_{0})\leq -\eta .
\end{equation*}
\end{lemma}

The above lemma is crucial in order to apply some compact perturbation
results (see Ducrot et al. \cite{Ducrot-Liu-Magal08}).

Since $A$ is a Hille-Yosida operator, we know that $A$ generates an
integrated semigroup $\left\{ S_{A}(t)\right\} _{t\geq 0}$ on $X,$ and $%
t\rightarrow S_{A}(t)\left( 
\begin{array}{c}
x \\ 
\varphi%
\end{array}%
\right) $ is an integrated solution of 
\begin{equation*}
\frac{dv(t)}{dt}=Av(t)+\left( 
\begin{array}{c}
x \\ 
\varphi%
\end{array}%
\right) ,t\geq 0,\text{ }v(0)=0.
\end{equation*}%
Since $S_{A}\left( t\right) $ is linear we have 
\begin{equation*}
S_{A}(t)\left( 
\begin{array}{c}
x \\ 
\varphi%
\end{array}%
\right) =S_{A}(t)\left( 
\begin{array}{c}
0_{\mathbb{R}^{n}} \\ 
\varphi%
\end{array}%
\right) +S_{A}(t)\left( 
\begin{array}{c}
x \\ 
0%
\end{array}%
\right) ,
\end{equation*}%
where 
\begin{equation*}
S_{A}(t)\left( 
\begin{array}{c}
0_{\mathbb{R}^{n}} \\ 
\varphi%
\end{array}%
\right) =\int_{0}^{t}T_{A_{0}}(l)\left( 
\begin{array}{c}
0_{\mathbb{R}^{n}} \\ 
\varphi%
\end{array}%
\right) dl
\end{equation*}%
and $S_{A}(t)\left( 
\begin{array}{c}
x \\ 
0%
\end{array}%
\right) $ is an integrated solution of 
\begin{equation*}
\frac{dv(t)}{dt}=Av(t)+\left( 
\begin{array}{c}
x \\ 
0%
\end{array}%
\right) ,t\geq 0,\text{ }v(0)=0.
\end{equation*}%
Therefore, by using Lemma \ref{LE3.1} with $h(t)=x$ and the above results$,$
we obtain the following result.\ 

\begin{lemma}
\label{LE4.3}The linear operator $A$ generates an integrated semigroup $%
\left\{ S_{A}(t)\right\} _{t\geq 0}$ on $X$. Moreover, we have the following explicit formula 
\begin{equation*}
S_{A}(t)\left( 
\begin{array}{c}
x \\ 
\varphi%
\end{array}%
\right) =\left( 
\begin{array}{c}
0_{\mathbb{R}^{n}} \\ 
\widehat{S}_{A}(t)\left( x,\varphi \right)%
\end{array}%
\right) ,\left( 
\begin{array}{c}
x \\ 
\varphi%
\end{array}%
\right) \in X,
\end{equation*}%
where $\widehat{S}_{A}(t)$ is the linear operator defined by%
\begin{equation*}
\widehat{S}_{A}(t)\left( x,\varphi \right) =\widehat{S}_{A}(t)\left(
0,\varphi \right) +\widehat{S}_{A}(t)\left( x,0\right)
\end{equation*}%
with 
\begin{equation*}
\widehat{S}_{A}(t)\left( 0,\varphi \right) \left( \theta \right)
=\int_{0}^{t}T\varphi +S(l)(\varphi )dl
\end{equation*}%
and 
\begin{equation*}
\widehat{S}_{A}(t)\left( x,0\right) \left( \theta \right) =\left\{ 
\begin{array}{l}
(t+\theta )x,\text{ if }t+\theta \geq 0, \\ 
0,\text{ if }t+\theta \leq 0.%
\end{array}%
\right.
\end{equation*}
\end{lemma}

Now we focus on the spectrums of $A$ and $A+L.$ Since $A$ is a Hille-Yosida
operator and $L$ is a bounded linear operator, we can get that $A+L$ is a
Hille-Yosida operator$.$ Moreover $\left( A+L\right) _{0}:D(\left(
A+L\right) _{0})\subset \overline{D(A)}\rightarrow \overline{D(A)}$ the part
of $A+L$ in $\overline{D(A)}$ is the linear operator defined by 
\begin{equation*}
\left( A+L\right) _{0}\left( 
\begin{array}{c}
0 \\ 
\varphi%
\end{array}%
\right) =\left( 
\begin{array}{c}
0 \\ 
\varphi ^{\prime }%
\end{array}%
\right) ,\forall \left( 
\begin{array}{c}
0 \\ 
\varphi%
\end{array}%
\right) \in D\left( \left( A+L\right) _{0}\right) ,
\end{equation*}%
where%
\begin{equation*}
D\left( \left( A+L\right) _{0}\right) =\left\{ \left( 
\begin{array}{c}
0 \\ 
\varphi%
\end{array}%
\right) \in \left\{ 0_{\mathbb{R}^{n}}\right\} \times BUC_{\eta
}^{1}:\varphi ^{\prime }(0)=\widehat{L}\left( \varphi \right) \right\} .
\end{equation*}%
By Lemma \ref{LE2.1} in this article and Lemma 2.1 in Magal and Ruan \cite%
{Magal-Ruan09a}, we know that 
\begin{equation*}
\sigma \left( A\right) =\sigma \left( A_{0}\right) \text{ and }\sigma \left(
A+L\right) =\sigma \left( \left( A+L\right) _{0}\right)
\end{equation*}%
since $LT_{A_{0}}(t)$ is compact for any $t>0.$

By using the main result by Thieme in \cite{Thieme97} or by Ducrot et al. 
\cite{Ducrot-Liu-Magal08}, one obtain the following lemma.\ 

\begin{lemma}
\label{LE4.4}The essential growth bound of $\left( A+L\right) _{0}$
satisfies 
\begin{equation*}
\omega _{0,ess}\left( \left( A+L\right) _{0}\right) \leq \omega
_{0,ess}(A_{0})\leq -\eta .
\end{equation*}
\end{lemma}

In the following lemma, we start specifying the point spectrum of $\left(
A+L\right) _{0}$. Let 
\begin{equation*}
\Omega :=\{\lambda \in \mathbb{C}:{Re}(\lambda )> -\eta \}.
\end{equation*}
We now apply some results taken from Engel and Nagel \cite{Engel-Nagel} and
Webb \cite{Webb85,Webb87}.

\begin{lemma}
\label{LE4.5}The point spectrum of $\left( A+L\right) _{0}$ is the set 
\begin{equation*}
\sigma \left( A+L\right) \cap \Omega =\sigma _{P}\left( \left( A+L\right)
_{0}\right) \cap \Omega =\left\{ \lambda \in \Omega :\det \left( \Delta
\left( \lambda \right) \right) =0\right\} ,
\end{equation*}%
where 
\begin{equation}
\Delta \left( \lambda \right) =\lambda I-\widehat{L}\left( e^{\lambda
.}I\right) \in \rm{M}_n(\mathbb{C}) .  \label{4.3}
\end{equation}
\end{lemma}
\begin{proof}
Let $\lambda \in \Omega $. Since $\omega _{0,ess}\left( \left( A+L\right)
_{0}\right) \leq -\eta $ it follows that 
\begin{equation*}
\sigma \left( A+L\right) _{0}\cap \Omega =\sigma _{P}\left( \left(
A+L\right) _{0}\right) \cap \Omega
\end{equation*}%
(see \cite{Engel-Nagel, Webb85,Webb87}). But $\lambda \in \sigma _{P}\left(
\left( A+L\right) _{0}\right) $ if and only if there exists $\left( 
\begin{array}{c}
0_{\mathbb{R}^{n}} \\ 
\varphi%
\end{array}%
\right) \in D\left( \left( A+L\right) _{0}\right) \setminus \left\{
0\right\} $ such that 
\begin{equation*}
\left( A+L\right) _{0}\left( 
\begin{array}{c}
0_{\mathbb{R}^{n}} \\ 
\varphi%
\end{array}%
\right) =\lambda \left( 
\begin{array}{c}
0_{\mathbb{R}^{n}} \\ 
\varphi%
\end{array}%
\right) .
\end{equation*}%
That is to say that $\lambda \in \sigma _{P}\left( \left( A+L\right)
_{0}\right) $ if and only if there exists $\varphi \in BUC_{\eta }^{1}\left(
(-\infty ,0],\mathbb{C}^{n}\right) \setminus \left\{ 0\right\} $ such that 
\begin{equation}
\varphi ^{\prime }\left( \theta \right) =\lambda \varphi \left( \theta
\right) ,\forall \theta \leq 0  \label{4.4}
\end{equation}%
and 
\begin{equation}
\varphi ^{\prime }(0)=\widehat{L}\left( \varphi \right) .  \label{4.5}
\end{equation}%
Equation (\ref{4.4}) is equivalent to 
\begin{equation}
\varphi \left( \theta \right) =e^{\lambda \theta }\varphi \left( 0\right)
,\forall \theta \leq 0.  \label{4.6}
\end{equation}%
Therefore, 
\begin{equation*}
\varphi \neq 0\Leftrightarrow \varphi \left( 0\right) \neq 0.
\end{equation*}%
By combining (\ref{4.5}) and (\ref{4.6}), we obtain%
\begin{equation*}
\lambda \varphi \left( 0\right) =\widehat{L}\left( e^{\lambda .}\varphi
(0)\right) .
\end{equation*}%
The proof is completed.
\end{proof}

From the discussion in this section, we obtain the following proposition.

\begin{proposition}
\label{PROP4.6}The linear operator $A+L:D(A)\rightarrow X$ is a Hille-Yosida
operator, and $\left( A+L\right) _{0}$ the part of $A+L$ in $\overline{D(A)}$
is the infinitesimal generator of a strongly continuous semigroup $\left\{
T_{\left( A+L\right) _{0}}(t)\right\} _{t\geq 0}$ of bounded linear
operators on $\overline{D(A)}$.
\end{proposition}

\section{Projectors on the eigenspaces}

%See also Magal and Ruan \cite[Proposition 3.14 and Theorem 3.15]% {Magal-Ruan09a}. 
Since $\omega _{0,ess}((A+L)_{0})\leq -\eta $, we obtain
that $\sigma \left( A+L\right) \cap \Omega $ is nonempty and finite and each 
$\lambda _{0}\in \sigma \left( A+L\right) \cap \Omega $ is a pole of $\left(
\lambda I-\left( A+L\right) \right) ^{-1}$of finite order $k_{0}\geq 1.$
This means that $\lambda _{0}$ is isolated in $\sigma \left( A+L\right) $
and the Laurent's expansion of the resolvent around $\lambda _{0}$ takes the
following form\ 
\begin{equation}
\left( \lambda I-\left( A+L\right) \right) ^{-1}=\sum_{n=-k_{0}}^{+\infty
}\left( \lambda -\lambda _{0}\right) ^{n}B_{n}^{\lambda _{0}}.  \label{5.1}
\end{equation}%
The bounded linear operator $B_{-1}^{\lambda _{0}}$ is the projector on the
generalized eigenspace of $\left( A+L\right) $ associated to $\lambda _{0}.$
The goal of this section is to provide a method to compute $B_{-1}^{\lambda
_{0}}.$

We remark that%
\begin{equation*}
\left( \lambda -\lambda _{0}\right) ^{k_{0}}\left( \lambda I-\left(
A+L\right) \right) ^{-1}=\sum_{m=0}^{+\infty }\left( \lambda -\lambda
_{0}\right) ^{m}B_{m-k_{0}}^{\lambda _{0}}.
\end{equation*}%
So we have the following approximation formula%
\begin{equation}
B_{-1}^{\lambda _{0}}=\lim_{\lambda \rightarrow \lambda _{0}}\frac{1}{\left(
k_{0}-1\right) !}\frac{d^{k_{0}-1}}{d\lambda ^{k_{0}-1}}\left( \left(
\lambda -\lambda _{0}\right) ^{k_{0}}\left( \lambda I-\left( A+L\right)
\right) ^{-1}\right) .  \label{5.2}
\end{equation}

In order to give an explicit formula for $B_{-1}^{\lambda _{0}}$, we need
the following results. The proof of the lemma is similar to the proof of
Lemma 4.1 in Liu Magal and Ruan \cite{Liu-Magal-Ruan08}.

\begin{lemma}
\label{LE5.1} For each $\lambda \in \rho \left( A+L\right) ,$ we have the
following explicit formula for the resolvent of $A+L$%
\begin{equation}
\begin{array}{l}
\left( \lambda I-\left( A+L\right) \right) ^{-1}\left( 
\begin{array}{c}
\alpha \\ 
\varphi%
\end{array}%
\right) =\left( 
\begin{array}{c}
0_{\mathbb{R}^{n}} \\ 
\psi%
\end{array}%
\right) \\ 
\Leftrightarrow \\ 
\psi \left( \theta \right) =\int_{\theta }^{0}e^{\lambda \left( \theta
-s\right) }\varphi \left( s\right) ds+e^{\lambda \theta }\Delta \left(
\lambda \right) ^{-1}\left[ 
\begin{array}{c}
\alpha +\varphi \left( 0\right) + \\ 
\widehat{L}\left( \int_{.}^{0}e^{\lambda \left( .-s\right) }\varphi \left(
s\right) ds\right)%
\end{array}%
\right] .%
\end{array}
\label{5.3}
\end{equation}%
Furthermore, we have that%
\begin{eqnarray*}
\sigma \left( A+L\right) \cap \Omega &=&\sigma \left( \left( A+L\right)
_{0}\right) \cap \Omega \\
&=&\sigma _{P}\left( \left( A+L\right) _{0}\right) \cap \Omega \\
&=&\left\{ \lambda \in \Omega :\det \left( \Delta \left( \lambda \right)
\right) =0\right\} .
\end{eqnarray*}
\end{lemma}

\bigskip Now we introduce the following linear operators $\digamma
_{1}:X_{0}\rightarrow \mathbb{R}^{n}$ and $\digamma _{2}:\mathbb{R}%
^{n}\rightarrow X_{0}$ defined by 
\begin{equation*}
\digamma _{1}\left( 
\begin{array}{c}
0_{\mathbb{R}^{n}} \\ 
\varphi%
\end{array}%
\right) =\varphi (0),\text{ }\digamma _{2}\alpha =\left( 
\begin{array}{c}
\alpha _{\mathbb{R}^{n}} \\ 
0%
\end{array}%
\right) .
\end{equation*}%
From Lemma \ref{LE5.1}, we have 
\begin{eqnarray*}
\digamma _{1}\left( \lambda I-\left( A+L\right) \right) ^{-1}\text{ }%
\digamma _{2}\alpha &=&\Delta \left( \lambda \right) ^{-1}\alpha , \\
\forall \lambda &\in &\{\lambda \in \rho (\left( A+L\right) ):{Re}(\lambda
)\geq \omega _{0,ess}\left( \left( A+L\right) _{0}\right) \}, \\
\forall \alpha &\in &\mathbb{R}^{n}.
\end{eqnarray*}%
Since $\lambda \rightarrow \left( \lambda I-\left( A+L\right) \right) ^{-1}$
is holomorphic from $\Omega $ into $\mathcal{L(}X\mathcal{)}$, we deduce
from the above formula that the map $\lambda \rightarrow \Delta \left(
\lambda \right) ^{-1}$ is holomorphic in $\Omega .$ We know that $\Delta
\left( \cdot \right) ^{-1}$ has only finite order poles with order $\widehat{%
k}_{0}\geq 1$. Therefore, $\Delta \left( \lambda \right) ^{-1}$ has the
Laurent's expansion around $\lambda _{0}$ and takes the following form%
\begin{equation*}
\Delta \left( \lambda \right) ^{-1}=\sum_{n=-\widehat{k}_{0}}^{+\infty
}\left( \lambda -\lambda _{0}\right) ^{n}\Delta _{n},\text{ }\Delta _{n}\in 
\mathcal{L(\mathbb{R}}^{n}\mathcal{)}.
\end{equation*}%
From the following lemma we know that $\widehat{k}_{0}=k_{0}.$

\begin{lemma}
\label{LE5.2}Let $\lambda _{0}\in \sigma \left( A+L\right) \cap \Omega .$
Then the following are equivalent

\begin{itemize}
\item[\rm{(i)}] $\lambda _{0}$ is a pole of order $k_{0}$ of $\left( \lambda
I-\left( A+L\right) \right) ^{-1};$

\item[\rm{(ii)}] $\lambda _{0}$ is a pole of order $k_{0}$ of $\Delta \left(
\lambda \right) ^{-1};$

\item[\rm{(iii)}] $\lim_{\lambda \rightarrow \lambda _{0}}\left( \lambda -\lambda
_{0}\right) ^{k_{0}}\Delta \left( \lambda \right) ^{-1}\neq 0,$ and $%
\lim_{\lambda \rightarrow \lambda _{0}}\left( \lambda -\lambda _{0}\right)
^{k_{0}+1}\Delta \left( \lambda \right) ^{-1}=0.$
\end{itemize}
\end{lemma}
\begin{proof}
\bigskip The proof follows trivially from the explicit formula of the
resolvent of $A+L\ $obtained in Lemma \ref{LE5.1}.
\end{proof}

\begin{lemma}
\label{LE5.3}The matrices $\Delta _{-1},...,\Delta _{-k_{0}}$ must satisfy 
\begin{equation*}
\Delta _{k_{0}}\left( \lambda _{0}\right) \left( 
\begin{array}{c}
\Delta _{-1} \\ 
\Delta _{-2} \\ 
\vdots \\ 
\Delta _{-k_{0}+1} \\ 
\Delta _{-k_{0}}%
\end{array}%
\right) =\left( 
\begin{array}{c}
0 \\ 
\vdots \\ 
0%
\end{array}%
\right)
\end{equation*}%
and 
\begin{equation*}
\left( 
\begin{array}{ccccc}
\Delta _{-k_{0}} & \Delta _{-k_{0}+1} & \cdots & \Delta _{-2} & \Delta _{-1}%
\end{array}%
\right) \Delta _{k_{0}}\left( \lambda _{0}\right) =\left( 
\begin{array}{ccc}
0 & \cdots & 0%
\end{array}%
\right) ,
\end{equation*}%
where 
\begin{equation*}
\Delta _{k_{0}}\left( \lambda _{0}\right) :=\left( 
\begin{array}{ccccc}
\Delta \left( \lambda _{0}\right) & \Delta ^{(1)}\left( \lambda _{0}\right)
& \Delta ^{(2)}\left( \lambda _{0}\right) /2! & \cdots & \Delta
^{(k_{0}-1)}\left( \lambda _{0}\right) /\left( k_{0}-1\right) ! \\ 
0 & \ddots & \ddots & \ddots & \vdots \\ 
\vdots & 0 & \ddots & \ddots & \Delta ^{(2)}\left( \lambda _{0}\right) /2!
\\ 
\vdots &  & \ddots & \ddots & \Delta ^{(1)}\left( \lambda _{0}\right) \\ 
0 & \cdots & \cdots & 0 & \Delta \left( \lambda _{0}\right)%
\end{array}%
\right) .
\end{equation*}
\end{lemma}

From the above results we can obtain the explicit formula for the projector $%
B_{-1}^{\lambda _{0}}$ on the generalized eigenspace associated to $\lambda
_{0}$ , which is given in the following proposition.

\begin{proposition}
\label{PROP5.4}Each $\lambda _{0}\in \sigma \left( \left( A+L\right) \right) 
$ with ${Re}(\lambda _{0})\geq \omega _{0,ess}\left( \left( A+L\right)
_{0}\right) $ is a pole of $\left( \lambda I-\left( A+L\right) \right) ^{-1}$%
of order $k_{0}\geq 1.$ Moreover $k_{0}$ is the only integer such that there
exists $\Delta _{-k_{0}}\in M_{n}\left( \mathbb{R}\right) $ with $\Delta
_{-k_{0}}\neq 0,$ such that 
\begin{equation*}
\Delta _{-k_{0}}=\lim_{\lambda \rightarrow \lambda _{0}}\left( \lambda
-\lambda _{0}\right) ^{k_{0}}\Delta \left( \lambda \right) ^{-1}.
\end{equation*}%
Furthermore the projector $B_{-1}^{\lambda _{0}}$ on the generalized
eigenspace of $(A+L)$ associated $\lambda _{0}$ is defined by the following
formula 
\begin{equation}
B_{-1}^{\lambda _{0}}\left( 
\begin{array}{c}
\alpha \\ 
\varphi%
\end{array}%
\right) =\left[ 
\begin{array}{c}
0_{\mathbb{R}^{n}} \\ 
\sum_{j=0}^{k_{0}-1}\frac{1}{j!}\Delta _{-1-j}L_{j}^{2}(\lambda _{0})\left( 
\begin{array}{c}
\alpha \\ 
\varphi%
\end{array}%
\right)%
\end{array}%
\right] ,  \label{5.4}
\end{equation}%
where 
\begin{equation*}
\Delta _{-j}=\lim_{\lambda \rightarrow \lambda _{0}}\frac{1}{\left(
k_{0}-j\right) !}\frac{d^{k_{0}-j}}{d\lambda ^{k_{0}-j}}\left( \left(
\lambda -\lambda _{0}\right) ^{k_{0}}\Delta \left( \lambda \right)
^{-1}\right) ,j=1,...,k_{0},
\end{equation*}%
\begin{equation*}
L_{0}^{2}\left( \lambda \right) \left( 
\begin{array}{c}
\alpha \\ 
\varphi%
\end{array}%
\right) =e^{\lambda \theta }\left[ \alpha +\varphi \left( 0\right) +\widehat{%
L}\left( \int_{.}^{0}e^{\lambda \left( .-s\right) }\varphi \left( s\right)
ds\right) \right] ,
\end{equation*}%
and 
\begin{eqnarray*}
L_{j}^{2}\left( \lambda \right) \left( 
\begin{array}{c}
\alpha \\ 
\varphi%
\end{array}%
\right) &=&\frac{d^{j}}{d\lambda ^{j}}\left[ L_{0}^{2}(\lambda )\left( 
\begin{array}{c}
\alpha \\ 
\varphi%
\end{array}%
\right) \right] \\
&=&\sum_{k=0}^{j}C_{j}^{k}\theta ^{k}e^{\lambda \theta }\frac{d^{j-k}}{%
d\lambda ^{j-k}}\left[ \alpha +\varphi \left( 0\right) +\widehat{L}\left(
\int_{.}^{0}e^{\lambda \left( .-s\right) }\varphi \left( s\right) ds\right) %
\right] ,\text{ }j\geq 1,
\end{eqnarray*}%
here%
\begin{eqnarray*}
&&\frac{d^{i}}{d\lambda ^{i}}\left[ \alpha +\varphi \left( 0\right) +%
\widehat{L}\left( \int_{.}^{0}e^{\lambda \left( .-s\right) }\varphi \left(
s\right) ds\right) \right] \\
&=&\widehat{L}\left( \int_{.}^{0}\left( .-s\right) ^{i}e^{\lambda \left(
.-s\right) }\varphi \left( s\right) ds\right) ,i\geq 1.
\end{eqnarray*}
\end{proposition}

\section{Projector for a simple eigenvalue}

For Hopf bifurcation it is useful to get the projector for a simple
eigenvalue. In this section we study this case, i.e. $\lambda _{0}$ is a
simple eigenvalue of $(A+L).$ That is to say that $\lambda _{0}$ is pole of
order $1$ of the resolvent of $(A+L),$ and the dimension of the eigenspace
of $(A+L)$ associated to the eigenvalue $\lambda _{0}$ is $1.$

We know that $\lambda _{0}$ is a pole of order $1$ of the resolvent of $%
(A+L) $ if and only if there exists $\Delta _{-1}\neq 0,$ such that 
\begin{equation*}
\Delta _{-1}=\lim_{\lambda \rightarrow \lambda _{0}}\left( \lambda -\lambda
_{0}\right) \Delta \left( \lambda \right) ^{-1}.
\end{equation*}%
From Lemma \ref{LE5.3}, we have $\Delta _{-1}\Delta \left( \lambda
_{0}\right) =\Delta \left( \lambda _{0}\right) \Delta _{-1}=0.$ Hence 
\begin{equation*}
\Delta _{-1}\left[ B+\widehat{L}\left( e^{\lambda _{0}.}I\right) \right] =%
\left[ B+\widehat{L}\left( e^{\lambda _{0}.}I\right) \right] \Delta
_{-1}=\lambda _{0}\Delta _{-1}.
\end{equation*}%
Therefore, if $\dim \left[ N\left( \Delta \left( \lambda _{0}\right) \right) %
\right] =1,$ the rank of $\Delta _{-1}$ is $1$ and the dimension of the
eigenspace of\ $(A+L)$ associated to $\lambda _{0}$ is $1.$ Conversely, if $%
\dim \left[ N\left( \Delta \left( \lambda _{0}\right) \right) \right] >1,$
it is readly checked that the eigenspace of $(A+L)$ associated to $\lambda
_{0}$ is 
\begin{equation*}
\left\{ \left( 
\begin{array}{c}
0 \\ 
e^{\lambda _{0}\theta }x%
\end{array}%
\right) :x\in N\left( \Delta \left( \lambda _{0}\right) \right) \right\} .
\end{equation*}%
and $\lambda _{0}$ is not simple.\ 

In that case, there exist $V_{\lambda _{0}},W_{\lambda _{0}}\in \mathbb{C}%
^{n}\setminus \left\{ 0\right\} ,$ such that 
\begin{equation}
W_{\lambda _{0}}^{T}\Delta \left( \lambda _{0}\right) =0,\text{ and }\Delta
\left( \lambda _{0}\right) V_{\lambda _{0}}=0.  \label{6.1}
\end{equation}%
Hence 
\begin{equation}
\Delta _{-1}=V_{\lambda _{0}}W_{\lambda _{0}}^{T}.  \label{6.2}
\end{equation}%
Moreover since $B_{-1}^{\lambda _{0}}$ is a projector, we should have $%
B_{-1}^{\lambda _{0}}B_{-1}^{\lambda _{0}}=B_{-1}^{\lambda _{0}},$ i.e.%
\begin{eqnarray}
&&e^{\lambda _{0}\theta }\Delta _{-1}\left[ 
\begin{array}{c}
\Delta _{-1}\left[ \alpha +\varphi \left( 0\right) +\widehat{L}\left(
\int_{.}^{0}e^{\lambda _{0}\left( .-s\right) }\varphi \left( s\right)
ds\right) \right] \\ 
+\widehat{L}\left( \int_{.}^{0}e^{\lambda _{0}.}\Delta _{-1}\left[ 
\begin{array}{c}
\alpha +\varphi \left( 0\right) \\ 
+\widehat{L}\left( \int_{.}^{0}e^{\lambda _{0}\left( .-l\right) }\varphi
\left( l\right) dl\right)%
\end{array}%
\right] ds\right)%
\end{array}%
\right]  \label{6.3} \\
&=&e^{\lambda _{0}\theta }\Delta _{-1}\left[ \alpha +\varphi \left( 0\right)
+\widehat{L}\left( \int_{.}^{0}e^{\lambda _{0}\left( .-s\right) }\varphi
\left( s\right) ds\right) \right] .  \notag
\end{eqnarray}%
Taking $\varphi =0$ in (\ref{6.3}), we obtain 
\begin{equation}
\Delta _{-1}=\Delta _{-1}\left[ I+\widehat{L}\left( \int_{.}^{0}e^{\lambda
_{0}.}ds\right) \right] \Delta _{-1}.  \label{6.4}
\end{equation}%
Taking $\alpha =0$ in (\ref{6.3}), we obtain 
\begin{eqnarray*}
&&\Delta _{-1}\left[ \varphi \left( 0\right) +\widehat{L}\left(
\int_{.}^{0}e^{\lambda _{0}\left( .-s\right) }\varphi \left( s\right)
ds\right) \right] \\
&=&\Delta _{-1}\Delta _{-1}\left[ \varphi \left( 0\right) +\widehat{L}\left(
\int_{.}^{0}e^{\lambda _{0}\left( .-s\right) }\varphi \left( s\right)
ds\right) \right] \\
&&+\Delta _{-1}\widehat{L}\left( \int_{.}^{0}e^{\lambda _{0}.}\Delta _{-1}%
\left[ \varphi \left( 0\right) +\widehat{L}\left( \int_{.}^{0}e^{\lambda
_{0}\left( .-s\right) }\varphi \left( s\right) ds\right) \right] ds\right) \\
&=&\Delta _{-1}\left[ I+\widehat{L}\left( \int_{.}^{0}e^{\lambda
_{0}.}ds\right) \right] \Delta _{-1}\left[ \varphi \left( 0\right) +\widehat{%
L}\left( \int_{.}^{0}e^{\lambda _{0}\left( .-s\right) }\varphi \left(
s\right) ds\right) \right] .
\end{eqnarray*}%
Therefore, we obtain the following corollary.\ 

\begin{corollary}
\label{CO6.1}$\lambda _{0}\in \sigma \left( \left( A+L\right) \right) $ is a
simple eigenvalue of $\left( A+L\right) $ if and only if 
\begin{equation*}
\lim_{\lambda \rightarrow \lambda _{0}}\left( \lambda -\lambda _{0}\right)
^{2}\Delta \left( \lambda \right) ^{-1}=0
\end{equation*}%
and 
\begin{equation*}
\dim \left[ N\left( \Delta \left( \lambda _{0}\right) \right) \right] =1.
\end{equation*}%
Moreover the projector on the eigenspace associated to $\lambda _{0}$ is 
\begin{equation}
B_{-1}^{\lambda _{0}}\left( 
\begin{array}{c}
\alpha \\ 
\varphi%
\end{array}%
\right) =\left[ 
\begin{array}{c}
0_{\mathbb{R}^{n}} \\ 
e^{\lambda _{0}\theta }\Delta _{-1}\left[ \alpha +\varphi \left( 0\right) +%
\widehat{L}\left( \int_{.}^{0}e^{\lambda _{0}\left( .-s\right) }\varphi
\left( s\right) ds\right) \right]%
\end{array}%
\right] ,  \label{6.5}
\end{equation}%
where 
\begin{equation*}
\Delta _{-1}=V_{\lambda _{0}}W_{\lambda _{0}}^{T}
\end{equation*}%
with $V_{\lambda _{0}},W_{\lambda _{0}}\in \mathbb{C}^{n}\setminus \left\{
0\right\} $ are two vectors satisfying (\ref{6.1}) and 
\begin{equation*}
\Delta _{-1}=\Delta _{-1}\left[ I+\widehat{L}\left( \int_{.}^{0}e^{\lambda
_{0}.}ds\right) \right] \Delta _{-1}.
\end{equation*}
\end{corollary}

\section{Nonlinear semiflow}

Remembering that $F:=L+G$ one may rewrite the abstract Cauchy problem (\ref%
{1.8}) as follows 
\begin{equation}
\frac{dU(t)x}{dt}=AU(t)x+F(U(t)x),\;\;t\geq 0,\;\;U(0)x=x:=\left( 
\begin{array}{c}
0 \\ 
\varphi%
\end{array}%
\right) \in X_{0}.  \label{7.1}
\end{equation}%
We shall investigate the properties of the semiflow generated by the mild
solution of (\ref{7.1}). Namely the continuous function $U(.)x:[0,\tau
]\rightarrow X_{0}$ satisfies the fixed point problem 
\begin{equation}
U(t)x=T_{A_{0}}(t)x+\left( S_{A}\diamond F(.,U(.,0)x)\right) (t),\;\;t\geq 0,
\label{7.2}
\end{equation}%
or equivalently 
\begin{equation}
U(t)x=x+A\int_{0}^{t}U(l)xdl+\int_{0}^{t}F(U(l)x)dl,\;\;t\geq 0.  \label{7.3}
\end{equation}%
By using Lemma \ref{LE3.1} with $h(t)=f(x_{t})$ we obtain the following
lemma.

\begin{lemma}
\label{LE7.1} $v \in C([0,\tau],X_0)$ is mild solution of the abstract
Cauchy problem (\ref{7.1}) is and only if 
\begin{equation*}
v(t)=\left( 
\begin{array}{c}
0 \\ 
u(t,.)%
\end{array}
\right)
\end{equation*}
with 
\begin{equation*}
u(t,\theta)=x\left( t+\theta \right),\forall t \geq 0, \forall \theta \leq 0,
\end{equation*}
and 
\begin{equation*}
x(t)=\left\{ 
\begin{array}{l}
\varphi(0)+\int_0^t f(x_s)ds, \text{ if } t \geq 0, \\ 
\varphi(t), \text{ if } t \leq 0.%
\end{array}
\right.
\end{equation*}
\end{lemma}

Lemma 7.1 is important since it shows that it is equivalent to consider the
mild solutions of the abstract Cauchy problem (\ref{7.1}) or the solutions
of the functional differential equation (\ref{1.1}). Therefore we can apply
a certain number of results obtained for abstract Cauchy problems (see
Thieme \cite{Thieme90a}, Magal \cite{Magal}, Magal and Ruan \cite%
{Magal-Ruan07, Magal-Ruan09a, Magal-Ruan09b, Magal-Ruan18}).

\begin{definition}
\label{DE7.2} Consider two maps $\chi : X_{0}\rightarrow \left( 0,+\infty %
\right] $ and $U:D_{\chi }\rightarrow X_{0},$ where 
\begin{equation*}
D_{\chi }=\left\{ (t,x)\in \left[ 0,+\infty \right) \times X_{0}:0 \leq t <
\chi \left(x\right) \right\} .
\end{equation*}
We will say that $U$ is \textit{a maximal (autonomous) semiflow} on $X_{0}$
if $U$ satisfies the following properties:

\begin{description}
\item[\rm{(i)}] $\chi \left(U(t)x\right) +t=\chi \left(x\right),\;\forall x\in
X_{0},\;\forall t\in \left[ 0,\chi \left( x\right) \right) .$

\item[\rm{(ii)}] $U(0)x=x, \forall x\in X_{0}.$

\item[\rm{(iii)}] $U(t-s)U(s)x=U(t)x, \; \forall x\in X_{0},\;\forall t,s\in %
\left[ 0,\chi \left( x\right) \right) $ with $t\geq s.$

\item[\rm{(iv)}] If $\chi \left( x\right) <+\infty ,$ then 
\begin{equation*}
\lim_{t\rightarrow \chi \left( x\right) ^{-}}\left\Vert U(t)x\right\Vert
=+\infty .
\end{equation*}
\end{description}
\end{definition}

Set 
\begin{equation*}
D=\left[ 0,+\infty \right) \times X_{0}.
\end{equation*}%
In order to present a theorem on the existence and uniqueness of solutions
to equation (\ref{7.1}), we make the following definition.

\begin{definition}
\label{DE7.3} We will say that $f: BUC_{\eta } \rightarrow \mathbb{R}^{n}$
is \textrm{Lipschitz on bounded sets}, if for each $\xi>0$ there exists a
constant $\kappa\left(\xi \right)$ satisfying 
\begin{equation*}
\left\Vert f(\varphi)-f(\psi)\right\Vert \leq \kappa \left(\xi \right)
\left\Vert \varphi-\psi \right\Vert
\end{equation*}%
whenever $\varphi,\psi \in BUC_{\eta}$ with $\left\Vert \varphi \right\Vert
\leq \xi$ and $\left\Vert \psi \right\Vert \leq \xi .$
\end{definition}

It is clear that if $f$ is Lipschitz on bounded sets so is $F$. Therefore we
can apply Theorem 5.2 in Magal and Ruan \cite{Magal-Ruan07}.

\begin{theorem}
\label{TH7.4} Assume that $f:BUC_{\eta }\rightarrow \mathbb{R}^{n}$ is
Lipschitz on bounded sets. Then there exist a map $\chi :X_{0}\rightarrow
\left( 0,+\infty \right] $ and a maximal non-autonomous semiflow $U:D_{\chi
}\rightarrow X_{0}$, such that for each $x\in X_{0}$, $U(.)x\in C\left( %
\left[ s,s+\chi \left( s,x\right) \right) ,X_{0}\right) $ is a unique
maximal solution of (\ref{7.2}). Moreover, the subset $D_{\chi }$ is open in 
$D$ and the map $\left( t,x\right) \rightarrow U(t)x$ is continuous from $%
D_{\chi }$ into $X_{0}$. Furthermore 
\begin{equation}
U(t)\left( 
\begin{array}{c}
0 \\ 
\varphi%
\end{array}%
\right) =\left( 
\begin{array}{c}
0 \\ 
\widehat{U}(t)(\varphi )%
\end{array}%
\right)  \label{7.4}
\end{equation}%
where 
\begin{equation}
\widehat{U}(t)(\varphi )(\theta )=x_{\varphi }(t+\theta )  \label{7.5}
\end{equation}%
and $x_{\varphi }:(-\infty ,\tau _{\varphi })\rightarrow \mathbb{R}^{n}$
(with $\tau _{\varphi }:=\chi \left( 
\begin{array}{c}
0 \\ 
\varphi%
\end{array}%
\right) \leq +\infty $) is the unique continuous function satisfying the
integral equation 
\begin{equation*}
x_{\varphi }(t)=\varphi (0)+\int_{0}^{t}f(x_{\varphi ,s})ds,\forall t\in
\lbrack 0,\chi \left( 
\begin{array}{c}
0 \\ 
\varphi%
\end{array}%
\right) ).
\end{equation*}%
Moreover if $\tau (\varphi ):=\chi \left( 
\begin{array}{c}
0 \\ 
\varphi%
\end{array}%
\right) <+\infty $ 
\begin{equation*}
\lim_{t\nearrow \tau (\varphi )^{-}}\left\vert x_{\varphi }(t)\right\vert
=+\infty .
\end{equation*}
\end{theorem}

\begin{remark}
\label{REM7.5} By using the identification $\varphi \rightarrow \left( 
\begin{array}{c}
0 \\ 
\varphi%
\end{array}
\right)$, we deduce that the maps $\tau$ and $\widehat{U}$ also define a
maximal semiflow on $BUC_{\eta}$.
\end{remark}

By using the results about (ACP) we can derive some extra result about the
global existence of solutions, the positiveness of solution (see Martin and
Smith \cite{Martin-Smith}). The following result is based on Proposition 3.5
in Magal and Ruan \cite{Magal-Ruan09b} (see also Example 3.6 in \cite%
{Magal-Ruan09b}).

\begin{proposition}
\label{PROP7.6} Assume that $f:BUC_{\eta }\rightarrow \mathbb{R}^{n}$ is
Lipschitz on bounded sets. Assume that for each constant $M>0$ we can find $%
\lambda >0$ such that 
\begin{equation*}
\lambda \varphi (0)+f(\varphi )\geq 0
\end{equation*}%
whenever $\varphi \geq 0$ and $\Vert \varphi \Vert \leq M$. Then for each $%
\varphi \geq 0$ we have 
\begin{equation}
x_{\varphi }(t)\geq 0,\forall t\in \lbrack 0,\tau (\varphi )).  \label{7.6}
\end{equation}
\end{proposition}

%We refer to \cite{Magal, Magal-Ruan07, Magal-Ruan09a, Magal-Ruan09b, Magal-Ruan18,Thieme90a} for more results about abstract Cauchy problem that could also be
%applied here.

Next we focus on the property of asymptotic smoothness of the semiflow since some
extra analysis is needed here. We now turn to the relative compactness of
the positive orbits. This type of properties are needed in particular to talk
about omega-limit sets of bounded positive orbit, or about global attractors
(see Hale \cite{Hale88}, Sell and You \cite{Sell-You}).

Recall that Kuratovsky's measure of non-compactness is defined by 
\begin{equation*}
\kappa\left(B \right)= inf \left\{ \varepsilon >0: B \text{ can be covered
by a finite number of balls of radius} \leq \varepsilon \right \}
\end{equation*}
whenever $B$ is a bounded subset of $X$.

For various properties of Kuratowskis measure of noncompactness, we refer to
Deimling [75], Martin [187], and Sell and You [233, Lemma 22.2]. Mention
that the every bounded orbit is relatively compact (and precompact).

\begin{lemma}
\label{LE7.7} Let $\left( X,\left\Vert .\right\Vert \right) $ be a Banach
space and $\kappa \left( .\right)$ the measure of non-compactness defined as
above.\ Then for any bounded subset $B$ and $\widehat{B}$ of $X,$ we have
the following properties:

\begin{itemize}
\item[\rm{(i)}] $\kappa \left( B\right) =0$ if and only if $\overline{B}$ is
compact;

\item[\rm{(ii)}] $\kappa \left( B\right) =\kappa \left( \overline{B}\right);$

\item[\rm{(iii)}] If $B\subset \widehat{B}$ then $\kappa \left( B\right) \leq
\kappa \left( \widehat{B}\right);$

\item[\rm{(iv)}] $\kappa \left( B+\widehat{B}\right) \leq \kappa \left( B\right)
+\kappa \left( \widehat{B}\right),$ where $B+\widehat{B}=\left\{ x+y:x\in B,
y\in \widehat{B}\right\}.$
\end{itemize}

Moreover for each bounded linear operators $T\in \mathcal{L}\left( X\right) $
we have 
\begin{equation}
\left\Vert T\right\Vert _{\mathrm{ess}}:=\kappa \left( T\left(
B_{X}(0,1)\right) \right) \leq \left\Vert T\right\Vert _{\mathcal{L}\left(
X\right) }.  \label{7.7}
\end{equation}
\end{lemma}

\begin{definition}
\label{DE7.8} We will say that a bounded subset $B \subset X_0$ is \textrm{%
positively invariant} by $U$ if 
\begin{equation*}
U(t,x) \in B, \forall t \geq 0, \forall x \in B.
\end{equation*}
We say that U is \textrm{asymptotically smooth} if every positively
invariant bounded set is attracted by a compact subset
\end{definition}

\begin{proposition}
\label{PROP7.9} The semiflow $U$ (or $\widehat{U}$) is asymptotically smooth.
\end{proposition}
\begin{proof}
Assume that $B \subset X_{0}$ is positive by $U$. The semiflow $U$ can be
rewritten as follows 
\begin{equation*}
U(t)\left( 
\begin{array}{c}
0_{\mathbb{R}^{n}} \\ 
\varphi%
\end{array}
\right) =\left( 
\begin{array}{c}
0_{\mathbb{R}^{n}} \\ 
T(t)(\varphi) +S(t)(\varphi )%
\end{array}
\right) ,
\end{equation*}
where 
\begin{equation*}
T(t)(\varphi)(\theta) =\left\{ 
\begin{array}{l}
\varphi (0)+\int_0^{t+\theta} f(x_{\varphi,s})ds,\text{ if }t+\theta \geq 0,
\\ 
\varphi (0),\text{ if }t+\theta \leq 0.%
\end{array}
\right.
\end{equation*}
and 
\begin{equation*}
S(t)(\varphi )(\theta )=\left\{ 
\begin{array}{l}
0,\text{ if }t+\theta \geq 0, \\ 
\varphi (t+\theta )-\varphi (0),\text{ if }t+\theta \leq 0.%
\end{array}%
\right.
\end{equation*}
By using Arzela-Ascoli theorem, we deduce that $T(t)B$ is a relatively
compact subset. Therefore by Lemma \ref{LE7.7} (d) and (a) we obtain 
\begin{equation*}
\kappa(U(t)B) \leq \kappa(T(t)B)+\kappa(S(t)B)= \kappa(S(t)B).
\end{equation*}
Assume that $B$ is contained into a ball of radius $r>0$. By using Lemma \ref%
{LE7.7} (c) we deduce that 
\begin{equation*}
\kappa(S(t)B) \leq \kappa(S(t)B(0,r)) \leq r \left\Vert S(t) \right\Vert _{%
\mathcal{L}\left( X\right) } \leq 2r e^{-\eta t}.
\end{equation*}
Therefore 
\begin{equation*}
\lim_{t \rightarrow +\infty} \kappa(U(t)B)=0,
\end{equation*}
and the result follows by Lemma 2.1-(a) in Magal and Zhao \cite{Magal-Zhao}.
\end{proof}

\section{The local stability of equilibria}

Assume that 
\begin{equation*}
f(0_{BUC_{\eta }})=0_{\mathbb{R}^{n}}.
\end{equation*}%
Then 
\begin{equation*}
\overline{x}(t)=0_{\mathbb{R}^{n}},\forall t\in \mathbb{R}
\end{equation*}%
is an equilibrium solution of the FDE (\ref{1.1}). Similarly, $0_{X}$ is an
equilibrium of the ACP (\ref{7.1}). Assume that $f$ is continuously
differentiable locally around $0.$ Then the linearized equation of the FDE (%
\ref{1.1}) around $0$ is defined by (\ref{1.3}). The linearized equation of
the ACP (\ref{7.1}) around $0$ is defined by 
\begin{equation}
\left\{ 
\begin{array}{l}
\dfrac{dU(t)x}{dt}=AU(t)x+L(U(t)x),\;\;\text{ for }t\geq 0, \\ 
U(0)x=x:=\left( 
\begin{array}{c}
0 \\ 
\varphi%
\end{array}%
\right) \in X_{0},%
\end{array}%
\right.  \label{8.1}
\end{equation}%
where $L=DF(0)$.

By combining Lemma \ref{LE4.2}, as a consequence of Proposition 7.1 in Magal
and Ruan \cite{Magal-Ruan09b} (see also Thieme \cite{Thieme90a}) we can
obtain the following stability theorem.

\begin{theorem}
\label{TH8.1} Assume that $f(0_{BUC_{\eta}})=0_{\mathbb{R}^{n}}$ and assume
that $f$ is continuously differentiable locally around $0_{BUC_{\eta}}$.

The equilibrium $0_{X}$ of the abstract Cauchy problem (\ref{7.1}) is
asymptotically stable if for each $\lambda \in \Omega $ 
\begin{equation}
\det \left( \Delta \left( \lambda \right) \right) =0\Rightarrow Re\left(
\lambda \right) <0.  \label{8.2}
\end{equation}%
More precisely, if the above condition is satisfied, we can find three
constants $M\geq 1$, $\delta >0$ and $\varepsilon >0$ such that 
\begin{equation}
\Vert U(t)x\Vert \leq Me^{-\delta t}\Vert x\Vert ,\forall t\geq 0,
\label{8.3}
\end{equation}%
whenever for each $x\in X_{0}$ with $\Vert x\Vert \leq \varepsilon $.
\end{theorem}

\section{Hopf bifurcation}

In this section we give a few comments and remarks concerning the
results obtained in this paper. In order to apply the center manifold
theorem to study Hopf bifurcation results for infinite delay differential
equations with parameter 
\begin{equation}
\left\{ 
\begin{array}{l}
\dfrac{dx(t)}{dt}=f(\mu ,x_{t}),\;\forall t\geq 0, \\ 
x\left( \theta \right) =\varphi \left( \theta \right) ,\forall \theta \leq 0%
\text{ with }\varphi \in BUC_{\eta },%
\end{array}%
\right.  \label{9.1}
\end{equation}%
where $\mu \in \mathbb{R}$, and $f:\mathbb{R}\times BUC_{\eta }\rightarrow 
\mathbb{R}^{n}$ is a $C^{k}$ map with $k\geq 4.$ We assume that $f(\mu
,0)=0,\forall \mu \in \mathbb{R}.$ As before, by setting $v(t)=\left( 
\begin{array}{c}
0 \\ 
x_{t}%
\end{array}%
\right) $ we can rewrite the delay differential equation (\ref{9.1}) as the
following abstract non-densely defined Cauchy problem on the Banach space $X=%
\mathbb{R}^{n}\times BUC_{\eta }$%
\begin{equation}
\frac{dv(t)}{dt}=Av(t)+F(\mu ,v(t)),\;t\geq 0,\;\;v(0)=\left( 
\begin{array}{c}
0_{\mathbb{R}^{n}} \\ 
\varphi%
\end{array}%
\right) \in \overline{D(A)},  \label{9.2}
\end{equation}%
where $A:D(A)\subset X\rightarrow X$ is defined in (\ref{1.7}) and $F:%
\mathbb{R}\times \overline{D(A)}\rightarrow X$ by 
\begin{equation*}
F\left( \mu ,\left( 
\begin{array}{c}
0_{\mathbb{R}^{n}} \\ 
\varphi%
\end{array}%
\right) \right) =\left( 
\begin{array}{c}
f(\mu ,\varphi ) \\ 
0_{BUC_{\eta }}%
\end{array}%
\right) .
\end{equation*}%
Set 
\begin{equation*}
L\left( \mu ,\left( 
\begin{array}{c}
0_{\mathbb{R}^{n}} \\ 
\psi%
\end{array}%
\right) \right) =\partial _{x}F\left( \mu ,0\right) \left( 
\begin{array}{c}
0_{\mathbb{R}^{n}} \\ 
\psi%
\end{array}%
\right) =\left( 
\begin{array}{c}
\partial _{\varphi }f(\mu ,0)\psi \\ 
0_{BUC_{\eta }}%
\end{array}%
\right) =:\left( 
\begin{array}{c}
\widehat{L}\left( \mu ,\psi \right) \\ 
0_{BUC_{\eta }}%
\end{array}%
\right) .
\end{equation*}%
System (\ref{9.1}) becomes%
\begin{equation}
\frac{dv(t)}{dt}=Av(t)+L(\mu ,v(t))+G(\mu ,v(t)),\;t\geq 0,\;\;v(0)=\left( 
\begin{array}{c}
0_{\mathbb{R}^{n}} \\ 
\varphi%
\end{array}%
\right) \in \overline{D(A)}  \label{9.3}
\end{equation}%
with%
\begin{equation*}
G(\mu ,v(t))=F\left( \mu ,v(t)\right) -L\left( \mu ,v(t)\right) .
\end{equation*}%
By section 4, we know that the linear operator $A+L\left( \mu ,.\right)
:D(A)\rightarrow X$ is a Hille-Yosida operator and $\omega _{0,ess}(\left(
A+L\left( \mu ,.\right) \right) _{0})\leq -\eta .$ Let 
\begin{equation*}
\Omega :=\{\lambda \in \mathbb{C}:{Re}(\lambda )>-\eta \}.
\end{equation*}%
The point spectrum of $\left( A+L\left( \mu ,.\right) \right) _{0}$ is the
set 
\begin{equation*}
\sigma \left( A+L\left( \mu ,.\right) \right) \cap \Omega =\sigma _{P}\left(
\left( A+L\left( \mu ,.\right) \right) _{0}\right) \cap \Omega =\left\{
\lambda \in \mathbb{C}:\det \left( \Delta \left( \mu ,\lambda \right)
\right) =0\right\} ,
\end{equation*}%
where 
\begin{equation*}
\Delta \left( \mu ,\lambda \right) :=\lambda I-\widehat{L}\left( \mu
,e^{\lambda .}I\right) .
\end{equation*}%
Hence, $A+L\left( \mu ,.\right) $ satisfies Assumptions 1.1, 1.2 and 1.3(c)
in \cite{Liu-Magal-Ruan11}. In order to apply the Hopf bifurcation theorem
obtained in \cite{Liu-Magal-Ruan11}, we need to make the following
assumption.

\begin{assumption}
\label{Assumption9.1} Let $\varepsilon >0$ and $
f\in C^{k}\left( \left( -\varepsilon ,\varepsilon \right) \times
B_{_{BUC_{\eta }}}\left( 0,\varepsilon \right) ;
\mathbb{R}^{n}\right) $ $\mathrm{for}$ some $k\geq 4.$ Assume that $\det
\left( \Delta \left( 0,\lambda \right) \right) \mathrm{=0}$ has a
simple purely imaginary root $\lambda _{0}=i\omega \neq 0$ and 
\begin{equation}
\left\{ \lambda \in \mathbb{C} :\det \left( \Delta \left( 0,\lambda \right) \right) =0\right\} \cap i 
\mathbb{R}=\left\{ i\omega ,-i\omega \right\} \text{.}  \label{9.4}
\end{equation}
 Moreover, assume that $\frac{dRe\left( \lambda (0)\right) }{d\mu } \neq \mathrm{0}$, where $\lambda (\mu )$ is the branch of
eigenvalues of $\det \Delta \left( \mu ,\lambda \right) =0$ through $i\omega $ at $\mu =0.$
\end{assumption}

By combining the results presented in the previous sections and by using the
same argument as in \cite{Liu-Magal-Ruan11} one may extend the Hopf
bifurcation theorem from finite to infinite delay differential equations.

\begin{theorem}
\label{TH9.2} Let Assumption \ref{Assumption9.1} be satisfied. Then there
exist $\varepsilon ^{\ast }>0$ and three $C^{k-1}$ maps, $\varepsilon
\rightarrow \mu (\varepsilon )$ from $\left( 0,\varepsilon ^{\ast }\right) $
into $\mathbb{R}$, $\varepsilon \rightarrow \varphi _{\varepsilon }$ from $%
\left( 0,\varepsilon ^{\ast }\right) $ into $BUC_{\eta }$, and $\varepsilon
\rightarrow \gamma \left( \varepsilon \right) $ from $\left( 0,\varepsilon
^{\ast }\right) $ into $\mathbb{R},$ such that for each $\varepsilon \in
\left( 0,\varepsilon ^{\ast }\right) $ there exists a $\gamma \left(
\varepsilon \right) $-periodic function $x_{\varepsilon }\in C^{k}\left( 
\mathbb{R}, \mathbb{R}^{n}\right) $, which is a solution of (\ref{9.1}) for the parameter value
equals $\mu (\varepsilon )$ and the initial value $\varphi $ equals $\varphi
_{\varepsilon }$.\ Moreover, we have the following properties

\begin{itemize}
\item[\rm{(i)}] There exist a neighborhood $N$ of $0$ in $%
%TCIMACRO{\U{211d} }%
%BeginExpansion
\mathbb{R}
%EndExpansion
^{n}$ and an open interval $I$ in $\mathbb{R}$ containing $0$ such that for $%
\widehat{\mu }\in I$ and any periodic solution $\widehat{x}(t)$ in $N$ with
minimal period $\widehat{\gamma }$ close to $\frac{2\pi }{\omega }$ of (\ref%
{9.1}) for the parameter value $\widehat{\mu },$ there exists $\varepsilon
\in (0,\varepsilon ^{\ast })$ such that $\widehat{x}(t)=x_{\varepsilon
}(t+\theta )$ (for some $\theta \in \left[ 0,\gamma \left( \varepsilon
\right) \right) $), $\mu (\varepsilon )=\widehat{\mu },$ and $\gamma \left(
\varepsilon \right) =\widehat{\gamma }.$

\item[\rm{(ii)}] The map $\varepsilon \rightarrow \mu (\varepsilon )$ is a $%
C^{k-1}$ function and 
\begin{equation*}
\mu (\varepsilon )=\sum_{n=1}^{[\frac{k-2}{2}]}\mu _{2n}\varepsilon
^{2n}+O(\varepsilon ^{k-1}),\forall \varepsilon \in \left( 0,\varepsilon
^{\ast }\right),
\end{equation*}
where $[\frac{k-2}{2}]$ is the integer part of $\frac{k-2}{2}.$

\item[\rm{(iii)}] The period $\gamma \left( \varepsilon \right) $ of $%
t\rightarrow u_{\varepsilon }(t)$ is a $C^{k-1}$ function and 
\begin{equation*}
\gamma \left( \varepsilon \right) =\frac{2\pi }{\omega }[1+\sum_{n=1}^{[%
\frac{k-2}{2}]}\gamma _{2n}\varepsilon ^{2n}]+O(\varepsilon ^{k-1}),\forall
\varepsilon \in \left( 0,\varepsilon ^{\ast }\right) ,
\end{equation*}%
where $\omega $ is the imaginary part of $\lambda \left( 0\right) $ defined
in Assumption \ref{Assumption9.1}.
\end{itemize}
\end{theorem}

Actually such a Hopf bifurcation result is based on the fact that such a
system has a local center manifold. The existence of a local center manifold
is a direct consequence of the center manifold theorem obtained in \cite%
{Magal-Ruan09a}. To conclude we would like to mention that it is also
possible to apply the normal form theory presented in \cite{Liu-Magal-Ruan14}
to infinite delay differential equations.

\end{document}